\documentclass[11pt]{amsart}

\usepackage[dvips]{graphicx}

\usepackage{graphicx}
\usepackage{amssymb}
\usepackage{amsmath}
\usepackage{color}
\usepackage{times}

\usepackage[unicode,bookmarks,colorlinks]{hyperref}
\hypersetup{
    linkcolor=brickred,
}

\definecolor{mahogany}{cmyk}{0, 0.77, 0.87, 0}
\definecolor{salmon}{cmyk}{0, 0.53, 0.38, 0}
\definecolor{melon}{cmyk}{0, 0.46, 0.50, 0}
\definecolor{yellowgreen}{cmyk}{0.44, 0, 0.74, 0}
\definecolor{brickred}{cmyk}{0, 0.89, 0.94, 0.28}
\definecolor{OliveGreen}{cmyk}{0.64, 0, 0.95, 0.40}
\definecolor{RawSienna}{cmyk}{0, 0.72, 1.0, 0.45}
\definecolor{ZurichRed}{rgb}{1, 0, 0} 

\usepackage{amsmath,amstext,amssymb,amsopn,amsthm}
\usepackage{amsmath,amssymb,amsthm}
\usepackage[mathscr]{eucal}

\newtheorem{theorem}{Theorem}[section]

\newtheorem{remark}[theorem]{Remark}

\numberwithin{equation}{section}

\newcommand{\bR}{\mathbb{R}}

\begin{document}

\title[Heat Trace of non-local operators]
{Heat Trace of non-local operators}

\author{Rodrigo Ba\~nuelos}\thanks{R. Ba\~nuelos is supported in part  by NSF Grant
\# 0603701-DMS}
\address{Department of Mathematics, Purdue University, West Lafayette, IN 47907, USA}
\email{banuelos@math.purdue.edu}
\author{Selma Y{\i}ld{\i}r{\i}m Yolcu}
\address{Department of Mathematics, Purdue University, West Lafayette, IN 47907, USA}
\email{syildir@math.purdue.edu}

\begin{abstract} This paper extends results of M. van den Berg on two-term asymptotics for the trace of Sch\"odinger operators when the Laplacian is replaced by non-local (integral) operators corresponding to rotationally symmetric stable processes and other closely related L\'evy processes. 
\end{abstract}

\maketitle

\tableofcontents

\section{Introduction} 

There is an extensive literature of trace formul{\ae} for heat kernels of the Laplacian and its Schr{\"o}dinger perturbations  in spectral and scattering theory. For instance, van den Berg \cite{vanden1} and later Ba\~nuelos and S{\'a} Barreto \cite{BanSaBar},  explicitly compute several  coefficients in the asymptotic expansion of the trace of the heat kernel of the Schr{\"o}dinger operator $-\Delta+V$  as $t\downarrow0$ for potentials $V\in \mathcal{S}(\bR^d)$, the class  of  rapidly decaying functions at infinity. In particular, in \cite{BanSaBar} a general formula is obtained for these coefficients by using elementary Fourier transform methods.  For applications of these results to problems in scattering theory, see \cite{BanSaBar} and references therein.  The results in \cite{BanSaBar}  were extended by  Donnelly  \cite{Donn} to certain compact Riemannian manifold.  Heat asymptotic results  have also been widely used in statistical mechanics, for details we refer the reader to the article \cite{Lieb} of E. H. Lieb on the second virial coefficient of a hard-sphere gas at low temperatures, and to the article  \cite{PenPenS} of M. D. Penrose, O. Penrose and G. Stell on the sticky spheres in quantum mechanics.  We also refer the reader to Datchev and Hezari's  overview article \cite{KirHez} for various other related spectral asymptotic results and applications.

For non-local (integral) operators which arise by replacing the Brownian motion with other L\'evy processes, many questions concerning spectral asymptotics are wide open at this point.  These include, for example, the Weyl two-term asymptotics for the spectral counting function of  the fractional Laplacian with Dirichlet boundary conditions and the McKean-Singer \cite{MckSin} result involving the Euler characteristic of the domain in the third-term asymptotics of the trace of the heat semigroup. On the other hand, a two-term trace asymptotics involving the volume of the domain and the surface area of the boundary is proved in Ba\~nuelos and  Kulczycki \cite{BanKul3} (for $C^{1,1}$ domains) and in Ba\~nuelos, Kulczycki and Siudeja \cite{BanKulSiu} (for Lipschitz domains) for the fractional Laplacian associated with symmetric stable processes.  These results are in parallel to the results of van den Berg \cite{vanden2} and Brown \cite{Brow} for the Laplacian. For further recent work on the Dirichlet case in domains of Euclidean space, see Frank and Geisinger, \cite{FraGei1} and \cite{FraGei2}.
 The purpose of this paper is to obtain analogues of the van den Berg result in \cite{vanden1} for the 
fractional Laplacian $\Delta^{\alpha/2}$ associated with symmetric $\alpha$-stable processes, $0<\alpha<2$, in ${\mathbb R}^d$ and for other closely related non-local operators corresponding to sums of stable processes and the relativistic Brownian motion.  These are all L\'evy processes which are obtained by subordination of Brownian motion. 

\section{Stable processes, statement of results}
Before we state our results precisely, we recall the basic definitions and elementary notions related to stable processes. Let $X_t$ be the $d$-dimensional symmetric $\alpha$--stable process of order $\alpha\in(0,2]$ in ${\mathbb R}^d$. The process $X_t$ has stationary independent increments and its transition density $p_t^{(\alpha)}(x,y)=p_t^{(\alpha)}(x-y), \;t>0, \;x, y\in {\mathbb R}^d$, is determined by its Fourier transform (characteristic function)
\begin{equation}\label{freekernel} 
e^{-t|\xi|^{\alpha}}=E(e^{i\xi\cdot X_t})= \int_{{\mathbb R}^d}e^{i\xi\cdot y}p_t^{(\alpha)}(y)dy,\qquad t>0, \quad \xi\in{\mathbb R}^d.
\end{equation} 
Denoting by $P^x$ and $E^x$ the probability and expectation, respectively,  of this process starting at $x$ we have that  any Borel subset $B\subset {\mathbb R}^d$, $x\in {\mathbb R}^d$, $t>0$,
$$P^x(X_t\in B)=\int_B p_t^{(\alpha)}(x-y)dy,$$
where 
\begin{equation}\label{trdenalphafourier}
p_t^{(\alpha)}(x)=\frac{1}{(2\pi)^d}\int_{{\mathbb R}^d}e^{-ix\cdot\xi}e^{-t|\xi|^{\alpha}}d\xi.
\end{equation}
Here we can also write 
\begin{equation}\label{subordination}
p_t^{(\alpha)}(x)=\int_0^{\infty}\frac{1}{(4\pi s)^{d/2}}e^{\frac{-|x|^2}{4s}} \eta^{\alpha/2}_t(s)\,ds,
\end{equation}
where $\eta^{\alpha/2}_t(s)$ is the density for the $\alpha/2$-stable subordinator. 
While explicit formula for the transition density of  symmetric $\alpha$-stable processes are only available for $\alpha=1$ (the Cauchy process) and $\alpha=2$ (the Brownian motion), these processes share many of the basic properties of Brownian motion.   It also follows trivially from \eqref{subordination} that $p_t^{(\alpha)}(x)$ is radial, symmetric and deceasing in $x$. Thus exactly as in the Brownian motion case we have
\begin{equation}\label{trdenest}
p_t^{(\alpha)}(x)\leq t^{-d/\alpha}p_1^{(\alpha)}(0).  
\end{equation}
In fact, if we denote by $ \omega_d$ the surface area of the unit sphere in $\bR^d$ we can compute $p_1^{(\alpha)}(0)$ more explicitly.  
\begin{eqnarray}\label{p(0)value}
p_1^{(\alpha)}(0)=
\frac{1}{(2\pi)^{d}}\int_{\bR^d}e^{-|\xi|^{\alpha}}d\xi
&=&  \frac{\omega_d}{(2 \pi)^d\alpha}
 \int_0^{\infty}e^{-s} s^{(\frac{d}{\alpha}  -1)} ds\nonumber\\
 &=& \frac{\omega_d \Gamma(d/\alpha)}{(2\pi)^d\alpha}. 
\end{eqnarray}
Throughout this paper we will deal with radial transition functions and use the notation 
$$
p_t^{(\alpha)}(x, y)=p_t^{(\alpha)}(x-y). 
$$
In particular, 
$p_t^{(\alpha)}(x, x)=p_t^{(\alpha)}(0).
$
Of  importance  for us in this paper is the scaling property of these processes.   
More precisely, by a simple change of variables in \eqref{trdenalphafourier}, we see that these processes are self-similar with scaling 
\begin{equation}\label{scaling}
p_t^{(\alpha)}(x,y)=t^{-d/\alpha}p_1^{(\alpha)}(t^{-1/\alpha}x,t^{-1/\alpha}y).
\end{equation}

The transition densities $p_t^{(\alpha)}(x)$  satisfy the following well-known two sided inequality valid for all $x\in \bR^d$ and $t>0$,
\begin{equation}\label{basicestiamte}
C_{\alpha, d}^{-1} \left( t^{-d/\alpha}\wedge\frac{t}{|x|^{d+\alpha}} \right) \leq p_t^{(\alpha)}(x)\leq C_{\alpha, d} \left( t^{-d/\alpha}\wedge\frac{t}{|x|^{d+\alpha}} \right),
\end{equation}
where the constant $C_{\alpha, d}$ only depends on $\alpha$ and $d$. Here and throughout the paper we use the notation $a\wedge b=\min\{a, b\}$ and $a\vee b=\max\{a, b\}$ for any $a, b\in \bR$. 

For rapidly decaying functions  $f\in{\mathcal S}({\bR^d})$, we have the semigroup of the stable processes defined via the  the Fourier inversion formula by 

\begin{eqnarray*}
T_tf(x)&=&E^x[f(X_t)]=E^0[f(X_t+x)]\\
&=&\int_{{\mathbb R}^d}f(x+y)p_t(dy)=p_t*f(x)\\
&=&\frac{1}{(2\pi)^d}\int_{{\mathbb R}^d}e^{ix\cdot\xi}e^{-t|\xi|^{\alpha}}{\widehat{f}}(\xi)d\xi.
\end{eqnarray*}

By differentiating this at $t=0$ we see that its  infinitesimal generator is $\Delta^{\alpha/2}$ in the sense that $\widehat{\Delta^{\alpha/2} f}(\xi)=-|\xi|^{\alpha}\widehat{f}(\xi)$.   This is a non-local operator such that for suitable test functions, including all functions in $f\in C_0^{\infty}(\bR^d)$, we can define it as the principle value integral 

\begin{eqnarray}
\Delta^{\alpha/2}f(x)={\mathcal A}_{d,\alpha}\lim_{\epsilon\to0^+}\int_{\{|y|>\epsilon\}}\frac{f(x+y)-f(x)}{|y|^{d+\alpha}}dy,\label{infgen}
\end{eqnarray}
where 
$${\mathcal A}_{d,\alpha}=\frac{\Gamma\left(\frac{d-\alpha}{2}\right)}{2^{\alpha}\pi^{d/2}\Gamma\left(\frac{\alpha}{2}\right)}.$$

We will denote by $H_0$ the fractional Laplacian operator  $\Delta^{\alpha/2}, \alpha\in(0,2]$  and by $H$ its Schr{\"o}dinger perturbation $H=\Delta^{\alpha/2}+V$, where $V\in L^{\infty}(\bR^d)$.  We let $e^{-tH}$ and $e^{-tH_0}$ be the associated heat semigroups and let $p_t^H$ and $p_t^{(\alpha)}$ denote their corresponding transition densities (heat kernels). In fact, $p_t^{(\alpha)}$ is just as in \eqref{freekernel} and the Feynman-Kac formula gives  

\begin{equation}\label{feynmankac}
p_t^H(x,y)=p_t^{(\alpha)}(x,y)E_{x, y}^t\left(e^{-\int_0^tV(X_s)ds}\right), 
\end{equation}
where $E_{x, y}^t$ is the expectation with respect to the stable process (bridge) 
starting at $x$ conditioned to be at $y$ at time $t$.  The main object of study in this paper is the trace difference 
\begin{eqnarray}\label{trace}
Tr(e^{-tH}-e^{-tH_0})&=&\int_{{\mathbb R}^d}(p_t^H(x,x)-p_t^{(\alpha)}(x,x))dx\nonumber\\
&=& p_t^{(\alpha)}(0)\int_{\bR^d} E_{x, x}^t\left(e^{-\int_0^tV(X_s)ds}-1\right)dx\\
 &=& t^{-d/\alpha}p_1^{(\alpha)}(0)\int_{\bR^d} E_{x, x}^t\left(e^{-\int_0^tV(X_s)ds}-1\right)dx\nonumber, 
\end{eqnarray}
where $p_1^{(\alpha)}(0)$ is the dimensional constant given by the right hand side of \eqref{p(0)value}.  

Before we go further let us observe that this quantity is well defined for all $t>0$, provided $V\in L^{\infty}(\bR^d)\cap L^{1}(\bR^d)$.  Indeed, the elementary inequality $|e^{z}-1|\leq |z| e^{|z|}$ immediately gives that 
\begin{equation*}
\left|\int_{\bR^d} E_{x, x}^t\left(e^{-\int_0^tV(X_s)ds}-1\right)dx\right|\leq 
e^{t\|V\|_{\infty}} \int_{\bR^d} E_{x, x}^t\left(\int_0^t|V(X_s)|ds\right)dx.
\end{equation*} 
However,
\begin{eqnarray}\label{L1norm}
E_{x, x}^t\left(\int_0^t|V(X_s)|ds\right)&=&\int_0^tE_{x, x}^t|V(X_s)| ds\\
&=&\int_0^t\int_{{\mathbb R}^d}\frac{p^{(\alpha)}_s(x,y)p^{(\alpha)}_{t-s}(y,x)}{p_t^{(\alpha)}(x,x)}|V(y)|dy ds.\nonumber
\end{eqnarray}
The  Chapman--Kolmogorov equations and the fact that $p_t^{(\alpha)}(x, x)=p_t^{(\alpha)}(0,0)$ give that 
$$
\int_{{\mathbb R}^d}\frac{p^{(\alpha)}_s(x,y)p^{(\alpha)}_{t-s}(y,x)}{p_t^{(\alpha)}(x,x)}dx=1 
$$
and hence 
\begin{equation}\label{oftenused} 
\int_{\bR^d} E_{x, x}^t\left(\int_0^t|V(X_s)|ds\right) dx=t\|V\|_1. 
\end{equation}
It follows then that 
\begin{equation}\label{generalbound}
\big| Tr(e^{-tH}-e^{-tH_0})\big| \leq t^{-d/\alpha+1}p_1^{(\alpha)}(0)\|V\|_1e^{t\|V\|_{\infty}},  
\end{equation}
valid for all $t>0$ and all potentials $V\in L^{\infty}(\bR^d)\cap L^{1}(\bR^d)$.  

The previous  argument also shows that for all potentials 
$V\in L^{\infty}(\bR^d)\cap L^{1}(\bR^d)$,  
\begin{equation}\label{traceassum}
 Tr\left(e^{-tH}-e^{-tH_0}\right)=p_t^{(\alpha)}(0)\sum_{k=1}^{\infty} \frac{(-1)^k}{k!}\int_{\bR^d} E_{x,x}^t\left(\int_0^t V(X_s)ds\right)^k dx,
\end{equation}
where the sum is absolutely convergent for all $t>0$. 

We have the following two Theorems which  parallel the results in van den Berg \cite{vanden1} for the Laplacian.  Note however, that here (ii) in Theorem \ref{theorem1} (as well as Theorem \ref{theorem2}) is more general as we do not assume that the potential is nonnegative.   

\begin{theorem} \label{theorem1} (i) Let $V:{\mathbb R}^d\to(-\infty, 0]$, $V\in L^{\infty}(\bR^d)\cap L^{1}(\bR^d)$. 
Then for all $t>0$ 
\begin{equation}\label{thm1}
p_t^{\alpha}(0)t\|V\|_1\leq Tr(e^{-tH}-e^{-tH_0})
\leq p_t^{(\alpha)}(0)\left(t\|V\|_1+\frac{1}{2}t^2\|V\|_1\|V\|_{\infty}e^{t\|V\|_{\infty}}\right).
\end{equation}
In particular 
\begin{eqnarray}
Tr(e^{-tH}-e^{-tH_0})&=&p_t^{(\alpha)}(0)\left(t\|V\|_1+{\mathcal O}(t^2)\right)\nonumber\\
&=&t^{-d/\alpha}p_1^{(\alpha)}(0)\left(t\|V\|_1+{\mathcal O}(t^2)\right), 
\end{eqnarray}
as $t\downarrow 0$. 

(ii) If we only assume that $V\in L^{\infty}(\bR^d)\cap L^{1}(\bR^d)$,  then for all $t>0$, 
\begin{equation}
\Big|Tr(e^{-tH}-e^{-tH_0})+p_t^{(\alpha)}(0)t\int_{\bR^d}V(x)dx\Big| \leq p_t^{(\alpha)}(0)Ct^2\|V\|_1\|V\|_{\infty}e^{t\|V\|_{\infty}}, 
\end{equation}
for some universal constant $C$.   From this we conclude that 
\begin{equation}\label{firstterm}
Tr(e^{-tH}-e^{-tH_0})=p_t^{(\alpha)}(0)\left(-t\int_{\bR^d}V(x)dx+{\mathcal O}(t^2)\right),
\end{equation}
as $t\downarrow 0$. 
\end{theorem}

\begin{theorem}\label{theorem2} Suppose $V\in L^{\infty}(\bR^d)\cap L^{1}(\bR^d)$ and that it is also uniformly H{\"o}lder continuous of order  $\gamma$ (there exists a constant $M\in(0,\infty)$ such that $|V(x)-V(y)|\leq M|x-y|^{\gamma}$, for all $x, y\in \bR^d$) with $0<\gamma< \alpha\wedge1$, whenever $0<\alpha\leq 1$, and with $0<\gamma\leq 1$, whenever $1<\alpha<2$.  
Then for all $t>0$,
\begin{eqnarray}
&&\left|\left(Tr(e^{-tH}-e^{-tH_0})\right)+p_t^{(\alpha)}(0)t\int_{\bR^d}V(x)dx
-p_t^{(\alpha)}(0)\frac{1}{2}t^2\int_{\bR^d}|V(x)|^2dx\right|\nonumber\\
&&\leq C_{\alpha,\gamma, d}\|V\|_1p_t^{(\alpha)}(0)\left(|V\|_{\infty}^2e^{t\|V\|_{\infty}} t^3+t^{\gamma/\alpha+2}\right),  
\end{eqnarray}
where the constant $C_{\alpha,\gamma, d}$ depends only on $\alpha$, $\gamma$ and $d$.  In particular, 
\begin{equation}\label{thm4}
Tr(e^{-tH}-e^{-tH_0})=p_t^{(\alpha)}(0)\left(-t\int_{\bR^d} V(x)dx+\frac{1}{2}t^2\int_{\bR^d}|V(x)|^2dx+{\mathcal O}(t^{\gamma/\alpha+2})\right), 
\end{equation}
as $t\downarrow 0$.  
\end{theorem}

Theorems \ref{theorem1} and \ref{theorem2} are proved in \S2.  In \S3, we provide extensions to other non-local operators. 
\section{Proof of Theorems  \ref{theorem1} and \ref{theorem2}}

\begin{proof}[Proof of Theorem \ref{theorem1}.]  Setting 
$$
a=\int_0^t V(X_s)ds,\quad \text{and}\quad 
b=t\|V\|_{\infty},
$$
we observe that $-b\leq a\leq 0$.  We then use the following elementary inequality from  \cite{vanden1}
\begin{equation}\label{ineqab1}
-a\leq e^{-a}-1\leq -a\left(1+\frac{1}{2}be^b\right)
\end{equation}
to conclude that 
\begin{equation}
-\int_0^t V(X_s)ds \leq  \left(e^{-\int_0^tV(X_s)ds}-1\right)\nonumber
\leq \left[{-\int_0^t V(X_s)ds}\right]\left(1+\frac{1}{2}t\|V\|_{\infty}e^{t\|V\|_{\infty}}\right).
\end{equation}
Taking expectations of both sides of this inequality with respect to $E_{x, x}^t$ and then integrating on $\bR^d$ with respect to $x$, it follows exactly as in the derivation of the general bound \eqref{generalbound}, that \eqref{thm1} holds.  This concludes the proof of (i) in Theorem \ref{theorem1}.

For (ii), we observe that by \eqref{traceassum} we have (again using \eqref{oftenused}) 
\begin{eqnarray}\label{part(ii)}
&&\Big|Tr(e^{-tH}-e^{-tH_0})+p_t^{(\alpha)}(0)t\int_{\bR^d}V(x)dx\Big|\\
&&\leq p_t^{(\alpha)}(0)\sum_{k=2}^{\infty} \frac{1}{k!}\int_{\bR^d} E_{x,x}^t\Big|\int_0^t V(X_s)ds\Big|^k dx\nonumber\\
&&\leq p_t^{(\alpha)}(0) \sum_{k=2}^{\infty} \frac{t^{k-1}\|V\|_{\infty}^{k-1}}{k!}\int_{\bR^d} E_{x,x}\left(\int_0^t |V(X_s)|ds\right) dx\nonumber\\
&&=p_t^{(\alpha)}(0)t\|V\|_1\sum_{k=2}^{\infty} \frac{t^{k-1}\|V\|_{\infty}^{k-1}}{k!}\leq Cp_t^{(\alpha)}(0)t^2\|V\|_1\|V\|_{\infty}e^{t\|V\|_{\infty}}\nonumber,
\end{eqnarray}
for some absolute constant $C$.  This proves (ii) and concludes the proof of the theorem. 
\end{proof}

\begin{remark}  We observe that the above holds for any operator which arises from subordination of a Brownian motion.  For such operators the transition probabilities $p_t$ are symmetric radial and decreasing, hence $p_t(x, x)=p_t(0)$.  In section \ref{othernonlocal}  below we give examples of other nonlocal operators for which Theorem \ref{theorem2} also holds. 
\end{remark}

\begin{proof}[Proof of Theorem \ref{theorem2}]  We begin by observing that (exactly as in the proof of inequality \eqref{part(ii)}) we have 
\begin{eqnarray}
&&\left|e^{-\int_0^t V(X_s)ds}-1 +\int_0^t V(X_s)ds-\frac{1}{2}\left[\int_0^t V(X_s)ds\right]^2\right|\nonumber\\
&\leq& C(t\|V\|_{\infty})^2e^{t\|V\|_{\infty}}\int_0^t|V(X_s)|ds,\label{ineqexpV2}
\end{eqnarray}
for some constant $C$. By taking expectation of both sides of \eqref{ineqexpV2} with respect to $E_{x, x}^t$ and then integrating with respect to $x$, we obtain  

\begin{eqnarray*}
&&\int_{\bR^d} E_{x,x}^t\left(\left|e^{-\int_0^t V(X_s)ds}-1+\int_0^t V(X_s)ds-\frac{1}{2}\left[\int_0^t V(X_s)ds\right]^2\right|\right)dx\nonumber\\
&\leq& C(t\|V\|_{\infty})^2e^{t\|V\|_{\infty}}\int_{\bR^d}E_{x,x}^t \left(\int_0^t|V(X_s)|ds\right)  dx\\
&=&C(t\|V\|_{\infty})^2e^{t\|V\|_{\infty}} t\|V\|_1,
\end{eqnarray*}
where we use again \eqref{oftenused}.  

Returning to the definition of the trace differences in \eqref{trace} we see that this leads to 
\begin{eqnarray}
&&\left|\frac{1}{p_t^{(\alpha)}(0)}(Tr(e^{-tH}-e^{-tH_0}))+t\int_{\bR^d}V(x)dx
-\frac{1}{2}\int_{{\mathbb R}^d}E_{x,x}^t\left(\left[\int_0^tV(X_s)ds\right]^2\right)dx\right|\nonumber\\
&& \;\;\;\;\;\; \leq C(t\|V\|_{\infty})^2e^{t\|V\|_{\infty}}t\|V\|_1.\label{tracedifference}
\end{eqnarray}
It remains to estimate the term $E_{x,x}^t([\cdot]^2)$. 
Since $V$ is uniformly H\"{o}lder with exponent $\gamma$ and constant $M$, we have
\begin{equation}\label{lipconst}
|V(X_s+x)-V(x)|\leq M|X_s|^{\gamma}.
\end{equation}
Hence, 
\begin{eqnarray}
\left|E_{x,x}^t\left[\int_0^tV(X_s)ds\right]^2-t^2V^2(x)\right|&=&\left|E_{x,x}^t\left[\int_0^tV(X_s)ds\right]^2-
\left[\int_0^tV(x)ds\right]^2\right|\nonumber\\
&=& \left|E_{x,x}^t\left(\left[\int_0^tV(X_s)ds\right]^2
-\left[\int_0^tV(x)ds\right]^2\right)\right|\nonumber\\
&=& E_{0,0}^t(\left[\int_0^t(V(X_s+x)-V(x))ds\right]\cdot\\
&&\left[\int_0^tV(X_s+x)+V(x))ds\right]).\nonumber
\end{eqnarray}
By employing \eqref{lipconst}, we obtain
$$\left|E_{x,x}^t\left(\left[\int_0^tV(X_s)ds\right]^2\right)-t^2V^2(x)\right|\leq  ME_{0,0}^t\left(\left[\int_0^t|X_s|^{\gamma}ds\right]\left[\int_0^t\left(|V(X_s+x)|+|V(x)|\right)ds\right]\right).$$
Integrating both sides of this inequality with respect to $x$ and using Fubini's theorem, we  have that the value of the second integral becomes $2t\|V\|_1$.  Thus we arrive at 
\begin{eqnarray}\label{expect^2}
&&\left|\int_{{\mathbb R}^d}E_{x,x}^t\left(\left[\int_0^tV(X_s)ds\right]^2\right)dx-t^2\int_{\bR^d}|V(x)|^2dx\right|\\
&\leq&2tM\|V\|_1E_{0,0}^t\left(\int_0^t|X_s|^{\gamma}ds\right)\nonumber.
\end{eqnarray}
Now, it remains to estimate the expectation on the right side of \eqref{expect^2}.  As in \eqref{L1norm} we have 
\begin{eqnarray}\label{conditional} 
E_{0,0}^t\left( \int_0^t |X_s|^{\gamma}ds\right)&=& \int_0^tE_{0,0}^t (|X_s|^{\gamma})ds\\
&=& \int_0^t\int_{{\mathbb R}^d}\frac{p_s^{(\alpha)}(0,y)p_{t-s}^{(\alpha)}(y,0)}{p_t^{(\alpha)}(0,0)}|y|^{\gamma}dyds\nonumber.
\end{eqnarray}

To estimate the right hand side we recall the following ``5P-inequality" 
\begin{equation}\label{5P}
\frac{p_s^{(\alpha)}(x,z)p_t^{(\alpha)}(z,y)}{p_{s+t}^{(\alpha)}(x,y)}\leq C_{\alpha, d}\left(p_s^{(\alpha)}(x,z)+p_t^{(\alpha)}(z,y)\right),
\end{equation}
valid for all $x, y, z\in \bR^d$, $s, t>0$ and $0<\alpha<2$,  where $C_{\alpha, d}$ is a constant depending only on $\alpha$ and $d$.  This inequality is proved by Bogdan and Jakubowski in \cite[p.182]{BogJak}. It is derived form the the ``3P-inequality" 
\begin{equation}\label{3P}
p_s^{(\alpha)}(x,z)\wedge p_t^{(\alpha)}(z,y)\leq C_{\alpha, d}\,p_{s+t}^{(\alpha)}(x,y).
\end{equation}

Inequality \eqref{3P} itself is proved by a simple computation replacing $p_s^{(\alpha)}$ with the quantity from \eqref{basicestiamte}.  From \eqref{3P} and the fact that for  $a, b\geq 0$ we have $ab = (a\wedge b)(a\vee b)$ and $(a\vee b)\leq (a + b)$), \eqref{5P} follows immediately.

Returning to \eqref{conditional}, we have, 
\begin{eqnarray*}
E_{0,0}^t\left( \int_0^t |X_s|^{\gamma}ds\right)
&\leq& C_{\alpha, d}\int_0^t\int_{{\mathbb R}^d}(p_s^{(\alpha)}(0,y)+p_{t-s}^{(\alpha)}(y,0)) |y|^{\gamma}dyds\\
&=& 2C_{\alpha, d}\int_0^t\int_{{\mathbb R}^d} p_s^{(\alpha)}(0,y)|y|^{\gamma}dyds\\
&=& 2C_{\alpha, d}\int_0^t E^0(|X_s|^{\gamma})ds\\
&=& 2C_{\alpha, d}\int_0^t s^{\gamma/\alpha}E^0(|X_1|^{\gamma})ds\\
&=& 2C_{\alpha, d}E^0(|X_1|^{\gamma})\frac{t^{\gamma/\alpha+1}}{\gamma/\alpha+1}, 
\end{eqnarray*}
where we used the scaling property of the stable process coming from \eqref{scaling}.  We now recall that $E^0(|X_1|^{\gamma})$  is finite under our assumption that $\gamma<\alpha$. (This fact follows trivially from \eqref{basicestiamte}.) Thus we see that 
\begin{equation}\label{bound}
E_{0,0}^t\left( \int_0^t |X_s|^{\gamma}ds\right)\leq  C_{\alpha,\gamma, d}\,t^{\gamma/\alpha+1},
\end{equation}
where the constant $C_{\alpha,\gamma, d}$ depends only on $\alpha$, $\gamma$ and $d$.   We conclude that 
\begin{eqnarray}\label{expection^2}
\left|\int_{{\mathbb R}^d}E_{x,x}^t\left(\left[\int_0^tV(X_s)ds\right]^2\right)dx-t^2\int_{\bR^d}|V(x)|^2dx\right|
&\leq&2tM\|V\|_1E_{0,0}^t\left(\int_0^t|X_s|^{\gamma}ds\right)\nonumber\\
&\leq & M\|V\|_1  C_{\alpha,\gamma, d}\,t^{\gamma/\alpha+2}.
\end{eqnarray}

Returning to \eqref{tracedifference} we find that 
\begin{eqnarray}\label{secondtracedifference}
&&\left|\frac{1}{p_t^{\alpha}(0)}\left(Tr(e^{-tH}-e^{-tH_0})\right)+t\int_{\bR^d}V(x)dx
-\frac{1}{2}t^2\int_{\bR^d}|V(x)|^2dx\right|\nonumber\\
&&\leq Ct^3\|V\|_{\infty}^2e^{t\|V\|_{\infty}}\|V\|_1
+M\|V\|_1  C_{\alpha,\gamma, d}\,t^{\gamma/\alpha+2}\\
&&\leq C_{\alpha,\gamma, d}\|V\|_1\left(|V\|_{\infty}^2e^{t\|V\|_{\infty}} t^3+t^{\gamma/\alpha+2}\right)\nonumber.
\end{eqnarray}
Rewriting this in the form stated in Theorem \ref{theorem2} we arrive at the announced bound 

\begin{eqnarray}
&&\left|\left(Tr(e^{-tH}-e^{-tH_0})\right)+p_t^{\alpha}(0)t\int_{\bR^d}V(x)dx
-p_t^{(\alpha)}(0)t^2\frac{1}{2}\int_{\bR^d}|V(x)^2dx\right|\nonumber\\
&&\leq C_{\alpha,\gamma, d}\|V\|_1p_t^{(\alpha)}(0)\left(|V\|_{\infty}^2e^{t\|V\|_{\infty}} t^3+t^{\gamma/\alpha+2}\right), 
\end{eqnarray}
valid  for all $t>0$. 
\end{proof}

\begin{remark}
We remark that in the case of Brownian motion, $\alpha=2$, one may use the fact (as done in \cite{vanden1}) that 
\begin{equation}\label{conditongaussian}
\frac{p_s^{(2)}(0,y)p_{t-s}^{(2)}(y,0)}{p_t^{(2)}(0,0)}=p_{\frac{s}{t}(t-s)}^{(2)}(0,y)
\end{equation}
to explicitly compute the right hand side of \eqref{conditional}.  While the inequality \eqref{5P} does not hold for $\alpha=2$, we may  estimate the right hand side of \eqref{conditional} in the case of Brownian motion without the explicit estimate that comes from \eqref{conditongaussian}. Indeed,  by observing that in law the Brownian bridge started at $0$ and conditioned to return to $0$ in time $t$ is the same as $B_s-\frac{s}{t}B_{t} $, where $B_s$ is standard Brownian motion,  a simple computation leads to the same conclusion as \eqref{bound} for $\alpha=2$.  Indeed,
\begin{eqnarray}\label{Brownian}
E_{0,0}^t\left( \int_0^t |B_s|^{\gamma}ds\right) &=&\int_0^t E^0|B_s-\frac{s}{t}B_{t}|^{\gamma} ds\nonumber\\
&\leq& C_{\gamma}\int_0^t E^0\left(|B_s|^{\gamma}+\left(\frac{s}{t}\right)^{\gamma}|B_t|^{\gamma}\right) ds\\
&=&C_{\gamma} E^0|B_1|^{\gamma} \int_0^{t}\left(s^{\gamma/2}+\left(\frac{s}{t}\right)^{\gamma}t^{\gamma/2}\right) ds\nonumber\\
&=& C_{\gamma, d}\, t^{\gamma/2+1},\nonumber 
\end{eqnarray}
where $C_{\gamma, d}$ depends only on $\gamma$ and $d$. 
  
Unfortunately, the simple path construction $X_s-\frac{s}{t}X_{t} $ only leads to the stable bridge when $\alpha=2$ and we are not able to repeat \eqref{Brownian} in the case $0<\alpha<2$, hence our use of \eqref{5P}.   However, it may be that the construction by Chaumont in  \cite{Cha} (or perhaps some of the estimates of Fitzsimmons and Getoor in \cite{FitGet}) can be used to bypass the estimate \eqref{5P}. 
 \end{remark}
 
 \section{Alternative  proofs of Theorems \ref{theorem1} and  \ref{theorem2}} 
 Recalling  formula  \eqref{traceassum} we have 
\begin{equation}\label{traceassum1}
 Tr\left(e^{-tH}-e^{-tH_0}\right)=p_t^{(\alpha)}(0)\sum_{k=1}^{\infty} \frac{(-1)^k}{k!}\int_{\bR^d} E_{x,x}^t\left(\int_0^t V(X_s)ds\right)^k dx,
\end{equation}
where as mentioned there the sum is absolutely convergent for all potentials $V\in L^{\infty}(\bR^d)\cap L^{1}(\bR^d)$ and all $t>0$.  We note also that in fact, for every $N\geq 1$, 

\begin{equation}\label{traceassum2} 
 Tr\left(e^{-tH}-e^{-tH_0}\right)=p_t^{(\alpha)}(0)\left(\sum_{k=1}^{N} \frac{(-1)^k}{k!}\int_{\bR^d} E_{x,x}^t\left(\int_0^t V(X_s)ds\right)^k dx + 0(t^{N+1})\right), 
\end{equation}
as $t\downarrow 0$. This follows from the fact that 
\begin{eqnarray}
\Big|\sum_{k=N+1}^{\infty} \frac{(-1)^{k}}{k!}\int_{\bR^d} E_{x,x}^t\left(\int_0^t V(X_s)ds\right)^k dx\Big|&\leq&
 \sum_{k=N+1}^{\infty} \frac{1}{k!}t^{k}\|V\|_{\infty}^{k-1}\|V\|_{1}\\
 &\leq& \frac{t^{N+1}\|V\|_{\infty}^N\|V\|_1}{(N+1)!} e^{t\|V\|_{\infty}}\nonumber
\end{eqnarray}

 Next observe that 
$$
\left(\int_0^t V(X_s)ds\right)^k=k! J_k
$$
where 
$$
J_k=\int_{0\leq s_1\leq s_2\leq \cdots \leq s_k\leq t} V(X_{s_1})ds_1 V(X_{s_2}) ds_2\cdots V(X_{s_k}) ds_k.  
$$
Recalling that the finite dimensional  distributions of the stable bridge are given by 
$$
P_{x,x}^t\{X_{s_1}\in dy_1, X_{s_2}\in dy_2, \dots , X_{s_k}\in dy_k\}=\frac{1}{p_t^{(\alpha)}(x,x)}{\prod_{j=0}^k p_{s_{j+1}-s_j}^{(\alpha)}(y_{j+1}-y_j)},
$$
where $y_0=y_{k+1}=x$, $s_0=0, s_{k+1}=t$, we can write 

\begin{equation}
 Tr\left(e^{-tH}-e^{-tH_0}\right)=\sum_{k=1}^{\infty}(-1)^k\int_{0\leq s_1\leq s_2\leq \cdots \leq s_k\leq t}\tilde{I}_k(s_1, s_2,
  \dots, s_k) ds_1ds_2\cdots ds_k,  
 \end{equation}
 where 
\begin{eqnarray*}
\tilde{I}_k(s_1, s_2,
  \dots, s_k) =\int_{(R^d)^{k+1}}\prod_{j=0}^k p_{s_{j+1}-s_j}^{(\alpha)}(y_{j+1}-y_j)\prod_{j=1}^kV(y_j)d{y_1}d{y_2}. 
\cdots d{y_k}dx. 
\end{eqnarray*}
 Here (as several times before) we used the fact that  ${p_t^{(\alpha)}(x,x)}={p_t^{(\alpha)}(0)}.$
 By Fubini's theorem we may integrate with respect to $x$ first and use the fact that  
 $$
 \int_{\bR^d} p_{s_1}^{(\alpha)}(y_1-x)p_{t-s_k}^{(\alpha)}(y_k-x)dx=p_{t-(s_k-s_1)}^{(\alpha)}(y_k-y_1)
 $$
 to arrive at 
\begin{equation}\label{multipleinttrace}
 Tr\left(e^{-tH}-e^{-tH_0}\right)=\sum_{k=1}^{\infty}(-1)^k\int_{0\leq s_1\leq s_2\leq \cdots \leq s_k\leq t}{I}_k(s_1, s_2,
  \dots, s_k) ds_1ds_2\cdots ds_k,  
 \end{equation}
 where 
\begin{eqnarray*}
{I}_k(s_1, s_2,
\dots, s_k) &=&\int_{(R^d)^{k}}\left(p_{t-(s_k-s_1)}^{(\alpha)}(y_k-y_1)V(y_1)\right)\cdot\\
&&\left(\prod_{j=1}^{k-1} p_{s_{j+1}-s_j}^{(\alpha)}(y_{j+1}-y_j)V(y_{j+1})\right)d{y_1}d{y_2}
\cdots d{y_k}.
\end{eqnarray*}
We can then also re-write the right hand side of \eqref{traceassum2} as 
\begin{equation}\label{traceassum3}
p_t^{(\alpha)}(0)\left(\sum_{k=1}^{N}\frac{(-1)^k}{p_t^{(\alpha)}(0)}\int_{0\leq s_1\leq s_2\leq \cdots \leq s_k\leq t}{I}_k(s_1, s_2,
  \dots, s_k) ds_1ds_2\cdots ds_k+ 0(t^{N+1})\right). 
\end{equation}

From  these formulas we can also compute coefficients in the asymptotic expansion of the trace as $t\downarrow0$ fairly directly and in particular give alternative proofs of Theorems \ref{theorem1} and \ref{theorem2}.   

If  $k=1$, we see that  
$$p_{t-(s_k-s_1)}^{(\alpha)}(y_k-y_1)=p_{t}^{(\alpha)}(0)$$ and 
$$
I_1(s_1)=p_{t}^{(\alpha)}(0)\int_{\bR^d} V(y)dy,  
$$
which immediately leads to \eqref{firstterm} in Theorem \ref{theorem1}.  

If  $k=2$, we write 
\begin{eqnarray}\label{secondterm1}
&&I_2(s_1, s_2) =\int_{\bR^d}\int_{\bR^d} p_{t-(s_2-s_1)}^{(\alpha)}(y_2-y_1)p_{(s_2-s_1)}^{(\alpha)}(y_2-y_1)V(y_1)V(y_2) dy_1dy_2\\
&=&\int_{\bR^d}\int_{\bR^d} p_{t-(s_2-s_1)}^{(\alpha)}(y_2-y_1)p_{(s_2-s_1)}^{(\alpha)}(y_2-y_1)\left(V(y_1)-V(y_2)+V(y_2)\right)V(y_2) dy_1dy_2\nonumber\\
&=& \int_{\bR^d}\int_{\bR^d} p_{t-(s_2-s_1)}^{(\alpha)}(y_2-y_1)p_{(s_2-s_1)}^{(\alpha)}(y_2-y_1)V^2(y_2) dy_1dy_2\nonumber\\
&+&\int_{\bR^d}\int_{\bR^d} p_{t-(s_2-s_1)}^{(\alpha)}(y_2-y_1)p_{(s_2-s_1)}^{(\alpha)}(y_2-y_1)\left(V(y_1)-V(y_2)\right)V(y_2)dy_1dy_2.\nonumber
\end{eqnarray}

Integrating first with respect to $y_1$  we conclude that 
$$
\int_{\bR^d}\int_{\bR^d} p_{t-(s_2-s_1)}^{(\alpha)}(y_2-y_1)p_{(s_2-s_1)}^{(\alpha)}(y_2-y_1)V^2(y_2) dy_1dy_2=p_{t}^{(\alpha)}(0)\int_{\bR^d}|V(y)|^2dy.
$$
On the other hand, $|V(y_1)-V(y_2)|\leq M |y_1-y_2|^{\gamma}$ an hence 
\begin{eqnarray}\label{secondtterm}
&&\left|\int_{\bR^d}\int_{\bR^d} p_{t-(s_2-s_1)}^{(\alpha)}(y_2-y_1)p_{(s_2-s_1)}^{(\alpha)}(y_2-y_1)\left(V(y_1)-V(y_2)\right)V(y_2)dy_1dy_2\right|\\
&\leq& M \int_{\bR^d}\int_{\bR^d} p_{t-(s_2-s_1)}^{(\alpha)}(y_2-y_1)p_{(s_2-s_1)}^{(\alpha)}(y_2-y_1)|y_1-y_2|^{\gamma}|V(y_2)|dy_1dy_2\nonumber\\
&=&M \int_{\bR^d}\left(\int_{\bR^d} p_{t-(s_2-s_1)}^{(\alpha)}(y_2-y_1)p_{(s_2-s_1)}^{(\alpha)}(y_2-y_1)|y_2-y_1|^{\gamma} dy_1\right) |V(y_2)|dy_2\nonumber\\
&=&M \left(\int_{\bR^d} p_{t-(s_2-s_1)}^{(\alpha)}(y)p_{(s_2-s_1)}^{(\alpha)}(y)|y|^{\gamma} dy \right) \left(\int_{\bR^d}|V(y)|dy\right).\nonumber
\end{eqnarray}

From inequality \eqref{5P} we again obtain 
$$
\int_{\bR^d} p_{t-(s_2-s_1)}^{(\alpha)}(y)p_{(s_2-s_1)}^{(\alpha)}(y)|y|^{\gamma} dy\leq 
C_{\alpha, d}\,p_t^{(\alpha)}(0)\int_{\bR^d}\left( p_{t-(s_2-s_1)}^{(\alpha)}(y)+p_{(s_2-s_1)}^{(\alpha)}(y)\right) |y|^{\gamma} dy.
$$
Thus 
$$
\int_{0\leq s_1\leq s_2\leq t}{I}_2(s_1, s_2) ds_1ds_2=p_{t}^{(\alpha)}(0)t^2\int_{\bR^d}|V(y)|^2dy+R_t
$$
where 
\begin{eqnarray}
|R_t| &\leq& C_{\alpha, d} M \|V\|_1p_{t}^{(\alpha)}(0)\int_{0\leq s_1\leq s_2\leq t}\left( E^{0}|X_{t-(s_2-s_1)}|^{\gamma}+E^{0}|X_{(s_2-s_1)}|^{\gamma}\right)ds_1ds_2\nonumber\\ 
&\leq & C_{\alpha, \gamma, d} M \|V\|_1p_{t}^{(\alpha)}(0) t^{\gamma/\alpha+2}.
\end{eqnarray}
As in \eqref{secondtracedifference}, this gives another proof of \eqref{thm4} in Theorem \ref{theorem2}. 
 
 \section{Extension to other non-local operators}\label{othernonlocal} 


Taking $0<\beta<\alpha<2$ and $a\geq 0$, we consider the process $Z_t^a=X_t+aY_t$, where $X_t$ and $Y_t$ are independent $\alpha$-stable and $\beta$-stable processes, respectively. This process is called the independent sum of the symmetric $\alpha$-stable process $X$ and the symmetric $\beta$-stable process $Y$ with weight $a$. The infinitesimal generator of $Z^a$ is $\Delta^{\alpha/2}+a^{\beta} \Delta^{\beta/2}$.   This again is a non-local operator.   Acting on functions $f\in C_0^{\infty}(\bR^d)$ we have 

\begin{eqnarray*}
\left(\Delta^{\alpha/2}+a^{\beta}\Delta^{\beta/2}\right)f(x)=
{\mathcal A}_{d,\alpha}\lim_{\epsilon\to0^+}\int_{\{|y|>\epsilon\}}\left(\frac{1}{|y|^{d+\alpha}}+\frac{a^{\beta}}{|y|^{d+\beta}}\right)[f(x+y)-f(x)]dy,
\end{eqnarray*}
where ${\mathcal A}_{d,\alpha}$ is as in \eqref{infgen}.

 Properties of the heat kernel (transition probabilities) for this operator have been studied by several people in recent years.  See for example Chen and Kumagai \cite{CheKum}, Chen, Kim and Song \cite{CheKimSon} or Jakubowski and Szczypkowski \cite{JakSzc}, and references given there. 
%
If we denote the heat kernel of this operator by $p_t^{a}(x)$, we have that 
 \begin{eqnarray}
 p_t^{a}(x)&=&\frac{1}{(2\pi)^d}\int_{{\mathbb R}^d}e^{-ix\cdot\xi}e^{-t(|\xi|^{\alpha}+a^{\beta}|\xi|^{\beta})}d\xi\\
 &=& \int_0^{\infty}\frac{1}{(4\pi s)^{d/2}}e^{\frac{-|x|^2}{4s}} \eta^{a}_t(s)\,ds,
 \end{eqnarray}
 where $\eta^{a}_t(s)$ is be the density function
of the sum of the $\alpha/2$-stable subordinator and $a^2$-times the $\beta/2$-stable subordinator. 
Again, this density is  radial, symmetric, and decreasing in $|x|$. It is proved in \cite{CheKum} and \cite{CheKimSon} that there is a constant $C_{\alpha, \beta, d}$ depending only on the parameters indicated such that for all $x\in \bR^d$ and $t>0$, 
\begin{equation}\label{chenkummixed} 
C_{\alpha, \beta, d}^{-1}f^a_t(x)\leq p^a_t(x)\leq C_{\alpha, \beta, d}f^a_t(x)
\end{equation}
where 
$$
f^a_t(x)=\left((a^\beta t)^{-d/\beta}\wedge t^{-d/\alpha}\right)\wedge\left(\frac{t}{|x|^{d+\alpha}}+\frac{a^{\beta}t}{|x|^{d+\beta}}\right).
$$
We note that for $a=0$ this is just the estimate in \eqref{basicestiamte}.  For any $\gamma>0$ with $0<\gamma<\beta<\alpha$ we have for any $t>0$, 
\begin{eqnarray}\label{mixedmoments}
\int_0^t E^0(|Z_s^a|^{\gamma})ds &\leq& C_{\gamma} \left(\int_0^t E^0(|X_s|^{\gamma})ds+ a^{\gamma}\int_0^t E^0(|Y_s|^{\gamma})ds\right)\nonumber\\
&=& C_{\gamma} \left(E^0(|X_1|^{\gamma})\frac{t^{\gamma/\alpha+1}}{\gamma/\alpha+1}+a^{\gamma}E^0(|Y_1|^{\gamma})\frac{t^{\gamma/\beta+1}}{\gamma/\beta+1}\right)\\
&=& C_{a, \alpha, \beta, d} \left(\frac{t^{\gamma/\alpha+1}}{\gamma/\alpha+1}+\frac{t^{\gamma/\beta+1}}{\gamma/\beta+1}\right)\nonumber.  
\end{eqnarray}
It also follows from the ``3P-inequality"  (24) in   \cite{JakSzc} (using again the fact that for any $a, b\geq 0$ we have $ab = (a\wedge b)(a\vee b)$ and $(a\vee b)\leq (a + b)$) that we have the following  ``5P-inequality" for this process: 

\begin{equation}\label{5P-a}
\frac{p_s^{a}(x,z)p_t^{a}(z,y)}{p_{s+t}^{a}(x,y)}\leq C_{a, \alpha, \beta, d}\left( p_s^{a}(x,z)+p_t^{a}(z,y) \right),
\end{equation}
valid for all $x,y, z\in \bR^d$ and all $s, t>0$.

With \eqref{mixedmoments} and \eqref{5P-a}, both proofs above for Theorem \ref{theorem2} can be repeated to obtain the same result for the operator $\Delta^{\alpha/2}+a^{\beta}\Delta^{\beta/2}$, replacing $\Delta^{\alpha/2}$.   More precisely we have 
   
 \begin{theorem}\label{theorem3}
Let $a\geq 0$, $0<\beta<\alpha <2$ and let $H_0^a=\Delta^{\alpha/2}+a^{\beta}\Delta^{\beta/2}$. Suppose $V\in L^{\infty}(\bR^d)\cap L^{1}(\bR^d)$ and that it is also uniformly H{\"o}lder continuous of order  $\gamma$,  with $0<\gamma< \beta\wedge1$, whenever $0<\beta \leq 1$, and with $0<\gamma\leq 1$, whenever $1<\beta<2$.
Let $H^{a}=\Delta^{\alpha/2}+a^{\beta}\Delta^{\beta/2}+V$. 
Then for all $t>0$,
\begin{eqnarray}
&&\left|\left(Tr(e^{-tH^a}-e^{-tH_0^a})\right)+p_t^{a}(0)t\int_{\bR^d}V(x)dx
-p_t^{a}(0)\frac{1}{2}t^2\int_{\bR^d}|V(x)|^2dx\right|\nonumber\\
&&\leq C_{a, \alpha, \beta, \gamma, d}\|V\|_1p_t^{a}(0)\left(|V\|_{\infty}^2e^{t\|V\|_{\infty}} t^3+t^{\gamma/\alpha+2}+t^{\gamma/\beta+2}\right),  
\end{eqnarray}
where the constant $C_{a, \alpha,\beta, \gamma, d}$ depends only on $a$, $\alpha$, $\beta$, $\gamma$ and $d$.  In particular, 
\begin{equation}\label{thm2}
Tr(e^{-tH^a}-e^{-tH_0^a})=p_t^{a}(0)\left(-t\int_{\bR^d} V(x)dx+\frac{1}{2}t^2\int_{\bR^d}|V(x)|^2dx+{\mathcal O}(t^{\gamma/\alpha+2})\right), 
\end{equation}
as $t\downarrow 0$.  
\end{theorem}

We next consider another stochastic process associated with a non-local operator that has been of considerable interest to many people in recent years, the so called relativistic Brownian motion and its more general versions, the $\alpha$-stable relativistic process.  This is again a L\'evy process denote by $X^m_t$ with characteristic function 

\begin{equation}\label{relativistic}
e^{-t\left(\left(|\xi|^{2}+m^{2/\alpha}\right)^{\alpha/2}-m\right)}=E(e^{i\xi\cdot X_t^m})=\int_{{\mathbb R}^d}e^{i\xi\cdot y}p_t^{(m, \alpha)}(y)dy, 
\end{equation}
for any $m\geq 0$ and $0<\alpha<2$.  The case $\alpha=1$ is the relativistic Brownian motion and $m=0$ yields back the stable processes of order $\alpha$. Henceforth we assume that $m>0$.  As in the case of stable processes, $X^m_t$ is a subordination of Brownian motion and in fact 

\begin{eqnarray}
p_t^{(m, \alpha)}(x)=&=&\frac{1}{(2\pi)^d}\int_{{\mathbb R}^d}e^{-ix\cdot\xi}e^{-t\left(\left(|\xi|^{2}+m^{2/\alpha}\right)^{\alpha/2}-m\right)}d\xi\\
 &=& \int_0^{\infty}\frac{1}{(4\pi s)^{d/2}}e^{\frac{-|x|^2}{4s}} \eta^{m,\alpha}_t(s)\,ds,\nonumber
 \end{eqnarray}
 where $\eta^{m,\alpha}_t(s)$ is the density of the subordinator with Bernstein  function $\Phi(\lambda)=(\lambda+m^{2/\alpha})^{\alpha/2}-m$.  As before we see that  $p_t^{(m, \alpha)}(x)$ is a radial, symmetric, and decreasing in $|x|$.  By a simple change of variables it is easy to see that the transition densities have the scaling property
 \begin{equation}\label{m-scaling}
 p_t^{(m, \alpha)}(x)=m^{d/\alpha}p_{mt}^{(1, \alpha)}(m^{1/\alpha}x).  
 \end{equation}
These properties can all be found in Ryznar \cite{Ryz2} where it is also proved that 
\begin{equation}\label{relintemrsofstable}
p_t^{(m, \alpha)}(x)=e^{mt}\int_0^{\infty}\frac{1}{(4\pi s)^{d/2}}e^{\frac{-|x|^2}{4s}} e^{-m^{2/\alpha}s}\eta^{\alpha/2}_t(s)\,ds,
\end{equation}
where $\eta^{\alpha/2}_t(s)$ is the density for the $\alpha/2$-stable subordinator. By scaling 
$$\eta^{\alpha/2}_t(s)=  t^{-2/\alpha}\eta^{\alpha/2}_1(st^{-2/\alpha}).$$ 
Hence changing variables in \eqref{relintemrsofstable} leads to  
\begin{equation}
\lim_{t\downarrow 0}e^{-mt}t^{d/\alpha}p_t^{(m, \alpha)}(0)=\int_0^{\infty}\frac{1}{(4\pi s)^{d/2}}\eta^{\alpha/2}_1(s)\,ds=p_1^{\alpha}(0)=\frac{\omega_d \Gamma(d/\alpha)}{(2\pi)^d\alpha},
\end{equation}
where the last equality follows from \eqref{p(0)value}. 

The infinitesimal generator of $X_t^m$ is given by $m-\left(m^{2/\alpha}-\Delta\right)^{\alpha/2}$. The case $\alpha=1$ gives the generator $m-\sqrt{-\Delta+m^2}$ which is the 
free relativistic Hamiltonian. For more on these  notions we refer the reader to Carmona, Masters and  Simon \cite{CarMasSim}.

Estimates for the global  transition probabilities $p_{t}^{(m, \alpha)}(x)$ and their Dirichlet counterparts for various domains have been studied by many authors.  We refer the reader to \cite{Che1},  \cite{CheKimKum},  \cite{CheKimSon1}, \cite{CheKimSon2}, \cite{CheSon}, \cite{CheKum}, \cite{Ryz1}, \cite{Ryz2}. Here we recall the following small time estimate  from Chen, Kim and Song \cite[Theorem 2.1]{CheKimSon1}. There exists a constant $C_{\alpha, m, d}$ depending on $\alpha$, $m$ and $d$ such that for all $x\in \bR^d$ and all $t\in (0, 1]$,

\begin{equation}
C_{\alpha, m, d}^{-1}\, t^{-d/\alpha}\wedge \frac{t\,\Psi(m^{\frac{1}{\alpha}}|x|)}{|x|^{d+\alpha}}\leq p_t^{(m, \alpha)}(x)\leq C_{\alpha, m, d} \,t^{-d/\alpha}\wedge \frac{t\,\Psi(m^{\frac{1}{\alpha}}|x|)}{|x|^{d+\alpha}},
\end{equation}
where 
$$
\Psi(r)=2^{-(d+\alpha)}\Gamma\left(\frac{d+\alpha}{2}\right)^{-1}\int_0^{\infty} s^{\frac{d+\alpha}{2}-1} e^{-s/4} e^{-r^2/s} ds
$$
which is a decreasing function of $r^2$ with $\Psi(0)=1$, $\Psi(r)\leq 1$ and with 
$$
c_1^{-1}e^{-r} r^{(d+\alpha-1)/2}\leq \Psi(r)\leq c_1e^{-r} r^{(d+\alpha-1)/2},
$$
for all $r\geq 1$.  (See also  Chen and Song,\cite[p. 277]{CheSon}.)  Setting 
$$
{\tilde p}_t(x)=t^{-d/\alpha}\wedge \frac{t\,\Psi(m^{\frac{1}{\alpha}}|x|)}{|x|^{d+\alpha}}
$$
it follows trivially (exactly as in Bogdan and Jakubowski  \cite[p.182]{BogJak}) that 
\begin{equation}
{\tilde p}_s(x)\wedge {\tilde p}_{t-s}(x)\leq C{\tilde p}_{t}(0),
\end{equation}
for all $x\in \bR^d$ and $0<s<t\leq 1$.  Therefore,       
\begin{equation}\label{3P-inequality}
p_s^{(m,\alpha)}(0,x)\wedge p_{t-s}^{(m,\alpha)}(x,0)\leq C_{\alpha, m, d}\,p_{t}^{(m,\alpha)}(0,0)
\end{equation}
for all $x\in \bR^d$ and $0<s<t\leq 1$. This as before leads to the ``5P-inequality"

 \begin{equation}\label{5P-relativistic}
\frac{p_s^{(m,\alpha)}(x, 0)p_{t-s}^{(m,\alpha)}(0,x)}{p_{t}^{(m, \alpha)}(0,0)}\leq C_{\alpha, m, d}\left( p_s^{(m,\alpha)}(x,0)+p_{t-s}^{(m,\alpha)}(0,x) \right),
\end{equation}
valid for all $x\in \bR^d$, $0<s<t\leq 1$.  (Recall that $p_s^{(m,\alpha)}(x,y)=p_s^{(m,\alpha)}(x-y)$.) Even though this estimate is not as general as the previous ``5P-inequalities" above, it suffices for our needs. 

Next, we need to compute bounds for the $\gamma$ moments of  $X_s^m$. 
More precisely we want to understand the behavior of 
$\int_0^{t}E^0|X_r^m|^{\gamma} dr$, as  $t\downarrow 0$. Let us denote  the standard Brownian motion in $\bR^d$ by $B_s$ and recall that $E^0|B_{2s}|^{\gamma}=2^{\gamma/2}s^{\gamma/2}E^0|B_1|^{\gamma}$ and $2^{\gamma/2}E^0|B_1|^{\gamma}=C_{\gamma, d}$,  where the latter is an explicit and computable constant whose exact value is not relevant for our calculations here. Using this, \eqref{m-scaling} and \eqref{relintemrsofstable} we have 
\begin{eqnarray}\label{rel-moments}
\int_0^{t}E^0|X_r^m|^{\gamma}ds&=&m^{\frac{d\gamma}{\alpha}}\int_0^{t}E^0|X_{mr}^1|^{\gamma}dr
=m^{({\frac{d\gamma}{\alpha}}-1)}\int_0^{mt}E^0|X_r^1|^{\gamma}dr\nonumber\\
&=&m^{({\frac{d\gamma}{\alpha}}-1)}\int_0^{mt}\int_{\bR^d}|x|^{\gamma}p_r^{(1,\alpha)}(x)dx dr\\
&=& m^{({\frac{d\gamma}{\alpha}}-1)}\int_0^{mt}e^r\left(\int_0^{\infty} 
E^{0}|B_{2s}|^{\gamma} e^{-s} \eta^{\alpha/2}_r(s)ds\right)dr\nonumber\\
&=&C_{\gamma, d}\, m^{({\frac{d\gamma}{\alpha}}-1)}\int_0^{mt}e^r\left(\int_0^{\infty} 
s^{\gamma/2} e^{-s} \eta^{\alpha/2}_r(s)ds\right)dr,\nonumber\\
&=& C_{\gamma, d}\, m^{({\frac{d\gamma}{\alpha}}-1)}\int_0^{mt}e^r\left(\int_0^{\infty} 
s^{\gamma/2} e^{-s} r^{-2/\alpha}\eta^{\alpha/2}_1(r^{-2/\alpha}s)ds\right)dr\nonumber\\
&=&C_{\gamma, d}\, m^{({\frac{d\gamma}{\alpha}}-1)}\int_0^{mt}r^{\gamma/\alpha}e^r\left(\int_0^{\infty} 
s^{\gamma/2} e^{-sr^{2/\alpha}} \eta^{\alpha/2}_1(s)ds\right)dr,\nonumber
\end{eqnarray}
where we used the scaling property of the stable subordinator for the second to the last line and a change variables for the last line.  But

$$
\int_0^{\infty} 
s^{\gamma/2} e^{-sr^{2/\alpha}} \eta^{\alpha/2}_1(s)ds\leq \int_0^{\infty} 
s^{\gamma/2}  \eta^{\alpha/2}_1(s)ds =C_{\gamma}E^0|X_1|^{\gamma}<\infty,
$$
whenever $\gamma<\alpha$.  (Here, $X_1$ is the stable process of order $\alpha$ at time 1.)  Thus, 
$$
\int_0^{t}E^0|X_r^m|^{\gamma}dr\leq C_{\alpha, \gamma, m, d}\, t^{\gamma/\alpha+1}e^{mt}  
$$
and this is valid for all $t>0$.  From this and the inequality \eqref{5P-relativistic}, we obtain 

\begin{eqnarray}\label{m-cond-bound}
E_{0,0}^t\left( \int_0^t |X^m_s|^{\gamma}ds\right)
&\leq& C_{\alpha,m, d}\int_0^t\int_{{\mathbb R}^d}(p_s^{(m,\alpha)}(y,0)+p_{t-s}^{(m,\alpha)}(0,y)) |y|^{\gamma}dyds\nonumber\\
&=& 2C_{\alpha, m, d}\int_0^t\int_{{\mathbb R}^d} p_s^{(m,\alpha)}(0,y)|y|^{\gamma}dyds\\
&=& 2C_{\alpha, m, d}\int_0^t E^0(|X^m_s|^{\gamma})ds\nonumber\\
&\leq& C_{\alpha,\gamma, m, d}\,t^{\gamma/\alpha+1},\nonumber
\end{eqnarray}
for all $t\in (0, 1]$. 

With \eqref{m-cond-bound}, the proof of Theorem \ref{theorem2} can be repeated to obtain  a similar result for the operator associated with the relativistic stable processes.  More precisely we have 
  
 \begin{theorem}\label{theorem4}
Let $H_0^m=m-\left(m^{2/\alpha}-\Delta\right)^{\alpha/2}$. Suppose $V\in L^{\infty}(\bR^d)\cap L^{1}(\bR^d)$ and that it is also uniformly H{\"o}lder continuous of order  $\gamma$,  with $0<\gamma< \alpha\wedge1$, whenever $0<\alpha\leq 1$, and with $0<\gamma\leq 1$, whenever $1<\alpha<2$.
Let $H^{m}=m-\left(m^{2/\alpha}-\Delta\right)^{\alpha/2}+V$. 
Then, for all $t\in (0, 1]$,
\begin{eqnarray}
&&\left|\left(Tr(e^{-tH^m}-e^{-tH_0^m})\right)+p_t^{\alpha}(0)t\int_{\bR^d}V(x)dx
-p_t^{\alpha}(0)\frac{1}{2}t^2\int_{\bR^d}|V(x)|^2dx\right|\nonumber\\
&&\leq C_{\alpha,\gamma, m, d}\|V\|_1p_t^{(m, \alpha)}(0)\left(|V\|_{\infty}^2e^{t\|V\|_{\infty}} t^3+t^{\gamma/\alpha+2}\right),
\end{eqnarray}
where the constant $C_{\alpha,\gamma, m, d}$ depends only on $\alpha$, $\gamma$, $m$ and $d$.  

In particular, 
\begin{equation}\label{thm2}
Tr(e^{-tH^m}-e^{-tH_0^m})=p_t^{(m, \alpha)}(0)\left(-t\int_{\bR^d} V(x)dx+\frac{1}{2}t^2\int_{\bR^d}|V(x)|^2dx+{\mathcal O}(t^{\gamma/\alpha+2})\right), 
\end{equation}
as $t\downarrow 0$.  
\end{theorem}

\end{document}